\documentclass[12pt, reqno]{amsart}
\usepackage{amsmath, amsthm, amscd, amsfonts, amssymb, graphicx, color,tikz}
\usepackage[bookmarksnumbered, colorlinks, plainpages]{hyperref}
\usepackage{mathrsfs}
\usepackage{enumerate}
\usepackage[utf8]{inputenc}
\usepackage{bbold}

\usepackage[all]{xy}  

\newtheorem{theorem}{Theorem}[section]
\newtheorem{lemma}[theorem]{Lemma}

\newtheorem{corollary}[theorem]{Corollary}
\theoremstyle{definition}
\newtheorem{definition}[theorem]{Definition}
\newtheorem{example}[theorem]{Example}

\newtheorem{remark}[theorem]{Remark}

\numberwithin{equation}{section}

\newcommand{\veps}{\varepsilon}
\newcommand{\CC}{\mathbb C}
\newcommand{\RR}{\mathbb R}
\newcommand{\DD}{\mathbb D}
\newcommand{\TT}{\mathbb T}
\newcommand{\ZZ}{\mathbb Z}
\newcommand{\NN}{\mathbb N}

\newcommand{\alg}{\mathrm{alg\,}}

\newcommand{\fcd}{\mathcal{F}({\mathbb{C}}^d)}
\newcommand{\vect}{\mathrm{span\,}}

\newcommand{\sccd}{\mathcal S_c(\CC^d)}

\allowdisplaybreaks[4]

\hypersetup{colorlinks=true,linkcolor=blue,citecolor=green}

\begin{document}
\setcounter{page}{1}

\title[Semigroup of composition operators]
{Semigroup of composition operators on the Fock space of several complex variables}
\date{\today}

\author[F. Bayart, M. Wang, X. Yao]{Fr\'ed\'eric Bayart, Maofa Wang and Xingxing Yao}

\address{Laboratoire de Math\'ematiques Blaise Pascal UMR 6620 CNRS, Universit\'e Clermont Auvergne, Campus universitaire des C\'ezeaux, 3 place Vasarely, 63178 Aubi\`ere Cedex, France.}
\email{frederic.bayart@uca.fr}

\address{School of Mathematics and Statistics, Wuhan University, Wuhan 430072, China.}
\email{mfwang.math@whu.edu.cn}

\address{School of Mathematics and Physics, Wuhan Institute of Technology, Wuhan 430205, China.}
\email{xxyao.math@wit.edu.cn}




\begin{abstract}
This paper considers discrete and continuous semigroups of composition operators on the Fock space of several complex variables. For discrete semigroups generated by a single operator, the asymptotic behaviour is analysed. Strongly continuous semigroups are studied, and in particular it is shown that the classical Berkson--Porta theorem fails in this context. As an application, we show that any bounded composition operator is strongly reducible on the Fock space and that it is mean-ergodic.
\end{abstract}
\maketitle


\section{Introduction}

Let $d\geq 1$. We endow $\mathbb{C}^{d}$ with the standard inner product $\langle z,w\rangle=\sum_{j=1}^{d}z_{j} \overline{w_{j}}$ and let us denote $|z|=\sqrt{\langle z,{z}\rangle}$.
The classical Fock space on $\mathbb{C}^{d}$ is the separable Hilbert space of holomorphic functions on $\CC^d$ defined by
\begin{equation*}
\fcd=\left\{f\in H(\mathbb{C}^{d}):  \|f\|^{2}:=\frac{1}{(2\pi)^{d}}\int_{\mathbb{C}^{d}}|f(z)|^{2}e^{-\frac{|z|^{2}}{2}}dA(z)<\infty\right\}.
\end{equation*}
For more details about Fock spaces, we refer to the monograph written by Zhu \cite{Zh}.

\smallskip

A holomorphic map $\varphi:\mathbb{C}^{d}\rightarrow\mathbb{C}^{d}$
induces a composition operator $C_\varphi$ defined by $C_{\varphi}f=f\circ\varphi$ for any  $f\in \fcd$.
Carswell et al. \cite{CMS} have characterized when $C_{\varphi}$ is a bounded or compact composition operator on $\fcd$, as follows.

\begin{theorem}\label{bound-com}
Let $\varphi:\mathbb{C}^{d}\rightarrow\mathbb{C}^{d}$ be a holomorphic map.
\begin{itemize}
  \item[(a)] $C_{\varphi}$ is bounded on $\fcd$ if and only if $\varphi(z)=Az+B$, where $A\in \mathcal M_d(\mathbb C)$, $\|A\|\leq1$, $B\in \mathbb{C}^{d}$,
and $\langle A\zeta,B\rangle=0$ for all $\zeta\in\mathbb{C}^{d}$ with $|A\zeta|=|\zeta|$.
  \item[(b)] $C_{\varphi}$ is compact on $\fcd$ if and only if $\varphi(z)=Az+B$, where $A\in \mathcal M_d(\mathbb C)$ and $\|A\|<1$, $B\in \mathbb{C}^{d}$.
\end{itemize}
\end{theorem}

For convenience, we denote by $\mathcal{S}_{c}(\mathbb{C}^{d})$ the above self-maps of $\mathbb{C}^{d}$ inducing bounded composition operators on $\fcd$.
Carswell et al. \cite{CMS} have also
computed the norm of $C_\varphi$:
\begin{equation}\label{norm}
  \|C_{\varphi}\|=\exp\left(\frac{|z_0|^2 -|Az_0|^2+|B|^2}{4}\right)
\end{equation}
where $z_{0}$ is any solution of $(I_{d}-A^{\ast}A)z=A^{\ast}B$.
Here, and in the sequel, we denote by $A^{\ast}$ the conjugate transpose of the matrix $A$.

The aim of this paper is to extend the results of \cite{ChP} to several complex variables. More precisely, we study the three following questions:
\begin{itemize}
\item Let $\varphi\in\sccd.$ When does $(C_\varphi^n)$ converge in $\mathcal L(\fcd)$?
\item Let $(\varphi_t)_{t\geq 0}$ be a family of holomorphic self-maps of $\CC^d$. When is $(\varphi_t)_{t\geq 0}$
a continuous semigroup of $\mathcal S_{c}(\CC^d)$?
\item Let $(\varphi_t)_{t\geq 0}$ be a family of holomorphic self-maps of $\CC^d$. When is $(C_{\varphi_t})_{t\geq 0}$ a strongly continuous semigroup on $\fcd$?
\end{itemize}

We completely answer these three questions in Theorem \ref{thm:iteration},  Theorem \ref{thm:continuoussemigroup} and Theorem \ref{thm:c0}. As a byproduct of our results, we show in Example \ref{example} that the classical Berkson--Porta theorem fails in our context, i.e., there is a continuous semigroup $(\varphi_t)_{t\geq 0}$ of $\mathcal S_c(\CC^d)$ such that $(C_{\varphi_t})_{t\geq 0}$ is not a strongly continuous semigroup on $\mathcal F(\CC^d).$ Furthermore, we compute the infinitesimal generators of these semigroups. Finally we show that any bounded composition operator is strongly reducible on the Fock space (see Theorem \ref{fc-ness}), and we characterize mean and uniformly mean ergodic composition operators on $\fcd$ (see Theorem \ref{thm:mean-ergodic}).

\medskip

To obtain some of our results, it will be easier to work with $\varphi$ written in a simpler form. Let us write $\varphi(z)=Az+B$ and let $P,Q$
be two unitary matrices. Then $C_P,\ C_Q$ define bounded (and even unitary) operators on $\mathcal F(\CC^d)$
and
\begin{equation}\label{eq:unitary}
C_Q\circ C_{Az+B}\circ C_P=C_{PAQz+PB}.
\end{equation}
The following lemma is Proposition 2.2 of \cite{BT}.

\begin{lemma}\label{simility}
Let $\varphi(z)=Az+B\in\sccd$ and {let $p$ be the sum of the algebraic multiplicities of the eigenvalues of $A$ of modulus strictly less than $1.$ Then there exists a unitary matrix $P\in\mathcal M_d(\CC)$ such that,} setting $S=P^*AP$ and $v=P^* B,$
\begin{equation*}
  S=\begin{pmatrix} T & 0\\
  0 & U \end{pmatrix},
\end{equation*}
where $T\in \mathcal M_p(\mathbb C)$ is an upper triangular matrix such that its diagonal contains all the eigenvalues of $A$ of modulus strictly less than 1 and $U\in \mathcal M_{d-p}(\mathbb C)$ is a diagonal matrix containing all
the eigenvalues of $A$ of modulus $1$, $v\in \mathbb{C}^{p}\times \{0\}^{d-p}$ and $C_{\varphi}$ is similar to $C_{Sz+v}$, more precisely $C_\varphi=C_{P^*}\circ C_{Sz+v}\circ C_{P}.$
\end{lemma}

\section{Iteration} \label{sec:iteration}

In this section, we prove the following result, which extends \cite[Theorem 2.2]{ChP}.

\begin{theorem}\label{thm:iteration}
{Let $\varphi(z)=Az+B\in\sccd$, let $p$ be the sum of the algebraic multiplicities of the eigenvalues of $A$ of modulus strictly less than $1$.}
Let also $\sigma(A)$ be the spectrum of $A$ and let $P$, $T$ be the matrices obtained in Lemma \ref{simility}.
\begin{enumerate}
\item[(a)] If $\sigma(A)\cap(\TT\backslash\{1\})\neq\varnothing,$ the sequence $(C_\varphi^n)$ does not converge for the weak operator topology.
\item[(b)] If $\sigma(A)\cap(\TT\backslash\{1\})=\varnothing,$ then $(C_\varphi^n)$ converges in norm to $C_\psi,$
with $\psi(z)=Qz+R$ where
$$Q=P\widetilde Q P^*\textrm{ where }\widetilde Q=\begin{pmatrix} 0&0\\0&I_{d-p}\end{pmatrix}$$
$$R=(I_{d}-PNP^*)^{-1}B\textrm{ where }N=\begin{pmatrix} T&0\\0&0\end{pmatrix}.$$
\end{enumerate}
\end{theorem}

\begin{proof}
We may and shall assume that $A$ and $B$ have the form given by Lemma \ref{simility}. Assume first that there exists
$\lambda \in\sigma(A) \cap(\TT\backslash \{1\}).$ Then $\lambda$ is a diagonal entry of $U$, the corresponding entry of $B$ is equal to $0,$ and we may assume that $C_\varphi(z_d)=\lambda z_d, $ so that $C_\varphi^{n}(z_d)=\lambda^n z_d$ for all $n\geq 1.$ Therefore, the sequence $(C_\varphi^n)$ fails to converge for the weak operator topology.

Conversely, assume that $\sigma(A)\cap(\TT\backslash \{1\})=\varnothing,$ which yields $U=I_{d-p}.$ The $n$-th iterate of $\varphi$ is given by
$$\varphi^{[n]}(z)=A^n z+\sum_{k=0}^{n-1}A^k B.$$
Since $B\in\CC^p\times\{0\}^{d-p},$ we know that
\begin{equation*}
  \sum_{k=0}^{n-1}A^k B=\sum_{k=0}^{n-1}N^k B
=(I_d-N)^{-1}(I_d-N^n)B
\end{equation*}
(recall that $0\notin\sigma(I_{d}-N)$). Therefore, for any $f\in \fcd$ and for any $z\in\CC^d,$
$$f( \varphi^{[n]}(z))\xrightarrow{n\to\infty} f(\widetilde Qz+(I_d-N)^{-1}B).$$
Let $\psi(z)=\widetilde Q(z)+(I_d-N)^{-1}B$ and let us show that $\|C_{\varphi^{[n]}}-C_\psi\|\to 0.$
For this part of the proof, our argument does not follow \cite{ChP} where it was used that $\ker(I-C_\varphi)$ has dimension $1,$ which is not the case here if $1\in\sigma(A).$

Let $f\in\fcd,$ fix $w=(w_{p+1},\dots,w_d)\in\CC^{d-p}$ and consider $g_w(z)=f(z,w)\in\mathcal F(\CC^p)$
which satisfies
$$\|f\|^2=\frac1{(2\pi)^{d-p}}\int_{\CC^{d-p}}\|g_w\|_{\mathcal F(\CC^p)}^2 e^{-|w|^2/2}dA(w)$$
and, due to the form of $A$ and $B,$
\begin{equation*}
\|C_\varphi^n (f)-C_\psi(f)\|^2=\frac1{(2\pi)^{d-p}}\int_{\CC^{d-p}}\|C_{\varphi_0}^n(g_w)-C_{\psi_0}(g_w)\|_{\mathcal F(\CC^p)}^2e^{-|w|^2/2}dA(w),
\end{equation*}
where $\varphi_0:\CC^p\to\CC^p,\ z\mapsto Tz+B_0,$ $\psi_0:\CC^p\to\CC^p,\ z\mapsto (I_p-T)^{-1}B_0,$
with $B_0$ the projection of $B$ onto $\vect(e_1,\dots,e_p).$ Write $\varphi_0^{[n]}(z)=T^n z+(I_p-T)^{-1}(I_p-T^n)B_0$ and set $B_n=(I_p-T)^{-1}(I_p-T^n)B_0,$ $F=(I_p-T)^{-1}B_0.$  Then
$$\begin{array}l
\displaystyle \|C_{\varphi_0}^n(g_w)-C_{\psi_0}(g_w)\|_{\mathcal F(\CC^p)}^2\\
\displaystyle \quad\quad=\frac1{(2\pi)^p}\int_{\CC^p}
|g_w(T^n z+B_n)-g_w(F)|^2e^{-|z|^2/2}dA(z)\\
\displaystyle\quad\quad =\frac{2^{2p}}{(2\pi)^p} \int_{\CC^p}
|g_w(2T^n \xi+B_n)-g_w(F)|^2e^{-2|\xi|^2} dA(\xi)
\end{array}
$$
where we have set $\xi=z/2.$
Let $\veps>0$ and let $K$ be a compact subset of $\CC^p$ such that, for all $\xi\notin K,$
$$ 2^{2p}e^{-2|\xi|^2} \leq \veps^2 e^{-|\xi|^2/2}.$$
Then
$$\begin{array}l
\displaystyle\frac{2^{2p}}{(2\pi)^p} \int_{\CC^p\backslash K}
|g_w(2T^n \xi+B_n)-g_w(F)|^2e^{-2|\xi|^2} dA(\xi)\\
\displaystyle \quad\quad \leq \frac{\veps^2} {(2\pi)^p}\int_{\CC^p}|g_w(2T^n \xi+B_n)-g_w(F)|^2 e^{-|\xi|^2/2}dA(\xi).
\end{array}$$
Let $n_0\in\mathbb N$ be such that $\|2T^n\|<1$ for $n\geq n_0,$ and for $n\geq n_0$, let $z_n$ be any solution to $(I_p-(2T^n)^*(2T^n))z_n=(2T^n)^*B_n$. Observe that since $(T^n)$ goes to $0$ and $(B_n)$ goes to $F,$ $(z_n)$ is bounded (it even goes to $0$).
Therefore \eqref{norm} shows that $(\|C_{\gamma_n}\|)_{n\geq n_0}$ is bounded, where $\gamma_n:\CC^p\to\CC^p,$ $\xi\mapsto 2T^n \xi+B_n$. Hence we get the existence of $C>0$ such that
$$ \frac{1} {(2\pi)^p}\int_{\CC^p}|g_w(2T^n \xi+B_n)-g_w(F)|^2 e^{-|\xi|^2/2}dA(\xi)\leq C \|g_w\|_{\mathcal F(\CC^p)}^2.$$
Moreover, since $K$ is compact, since $(T^n)$ goes to $0$ and $(B_n)$ goes to $F,$ an application of Cauchy's inequalities yields the existence of $n_1\in\mathbb N$ such that, for all $n\geq n_1,$
$$\frac{2^{2p}}{(2\pi)^p} \int_{K}
|g_w(2T^n \xi+B_n)-g_w(F)|^2 e^{-2|\xi|^2}dA(\xi)\leq \veps^2\|g_w\|_{\mathcal F(\CC^p)}^2.$$
Finally, we have found, for any $n\geq \max(n_0,n_1),$
$$\|C_\varphi^n (f)-C_\psi(f)\|^2\leq (C+1)\veps^2\|f\|^{2},$$
which completes the proof that $(C_{\varphi}^{n})$ converges to $C_\psi$ in $\mathcal L(\fcd).$
\end{proof}

\section{Semigroups of $\mathcal S_c(\CC^d)$}
In this section, we investigate continuous semigroups of maps belonging to $\mathcal S_c(\CC^d)$.
We first recall some terminology.

\begin{definition}
Let $(\varphi_{t})_{t\geq 0}$ be a family of holomorphic self-maps of $\CC^d$.
We say that $(\varphi_{t})_{t\geq0}$ is a \emph{continuous semigroup} if
\begin{itemize}
  \item[(i)] $\varphi_{0}(z)=z$, $z\in\CC^d$;
  \item[(ii)] For every $t,s\geq0$, $\varphi_{t}\circ\varphi_{s}=\varphi_{t+s}$;
  \item[(iii)] $\varphi_{t}\rightarrow\varphi_{0}$ uniformly on compact subsets of $\CC^d$ as $t\rightarrow0^{+}$.
\end{itemize}
\end{definition}

Following the pioneering work of Berkson and Porta in \cite{BP}, Abate \cite{Abate92} has shown that there exists a holomorphic map $G:\CC^d\to\CC^d$ such that $({\varphi_t})_{t\geq 0}$ is the solution of the Cauchy problem:
\begin{equation*}
\frac{\partial \varphi_t(z)}{\partial t}=G(\varphi_t(z)),\ t\in\RR_+,\ z\in\CC^d.
\end{equation*}
Then $G$ is called the \emph{infinitesimal generator} of
the continuous semigroup $\left(\varphi_t\right)_{t\geq0}$.

\smallskip

We start with a characterization of continuous semigroups of maps of $\mathcal S_c(\CC^d)$. {We first recall the Lumer-Phillips theorem \cite{LP} (or \cite[Section II.3.b]{EN06}):}
for all $E\in\mathcal M_d(\CC),$ $\|\exp(tE)\|\leq 1$ for all $t\geq 0$ if and only if $\sigma(E+E^*)\subset(-\infty,0]$. We will need the following lemmas.

\begin{lemma}\label{lem:sg1}
Let $E\in\mathcal M_d(\CC)$ with $\sigma(E+E^*)\subset (-\infty,0]$ and let $\xi \in \CC^d.$ Then
$$\Re e\langle E\xi,\xi\rangle=0\iff \xi\in\ker(E+E^*).$$
\end{lemma}

\begin{proof}
Since $E+E^*$ is hermitian, it is diagonalizable with real eigenvalues and orthogonal eigenspaces. We thus decompose $\xi$ as $\sum_{i=1}^n \xi_i,$ with $(E+E^*)\xi_i=\lambda_i \xi_i,$ $\lambda_1=0,$ $\lambda_2,\dots,\lambda_n<0$ and $\xi_i\perp \xi_j$ for $i\neq j.$ Therefore
\begin{equation*}
  2\Re e\langle E\xi,\xi\rangle=\langle (E+E^*)\xi,\xi\rangle
=\sum_{i=1}^n \lambda_i |\xi_i|^2\leq 0.
\end{equation*}
This is equal to $0$ if and only if $\xi=\xi_1,$ namely $\xi \in\ker(E+E^*).$
\end{proof}

\begin{lemma}\label{lem:sg2}
Let $E\in\mathcal M_d(\CC)$ with $\sigma(E+E^*)\subset (-\infty,0]$ and let $\xi \in \CC^d.$ The following assertions are equivalent.
\begin{enumerate}[(i)]
\item $\forall t_0>0,$ {$|\exp(t_0E)\xi|=|\xi|.$}
\item $\exists t_0>0,$ {$|\exp(t_0E)\xi|=|\xi|.$}
\item $\forall n\geq 0,$ $E^* E^n \xi=-E^{n+1}\xi.$
\item $\forall m,n\geq 0,$ $(E^*)^m E^n \xi=(-1)^m E^{m+n}\xi.$
\item $\forall t_0>0,$ $\exp(t_0E^*)\exp(t_0E)\xi=\xi.$
\end{enumerate}
Moreover, if $E$ is normal, this is equivalent to
\begin{enumerate}[(i)]
\setcounter{enumi}{5}
\item $E^*\xi=-E\xi.$
\end{enumerate}
\end{lemma}
\begin{proof}
$(ii)\implies (iii)$. Since $(\exp(tE))_{t\geq 0}$ is a contractive semigroup, {we get $|\exp(tE)\xi|=|\xi|$ for all $t\in[0,t_0]$.
We set $u(t)=|\exp(tE)\xi|^2$} and we differentiate $u$ to find
\begin{align*}
u'(t)&=\langle E\exp(tE)\xi,\exp(tE)\xi\rangle+\langle \exp(tE)\xi,E\exp(tE)\xi\rangle\\
&=2\Re e\langle E\exp(tE)\xi,\exp(tE)\xi\rangle
\end{align*}
which has to be equal to $0$ on $[0,t_0]$. By Lemma \ref{lem:sg1}, this implies
$$E^*(\exp(tE)\xi)=-E(\exp(tE)\xi)$$
which yields, for all $t\in[0,t_0],$
$$\sum_{n=0}^{+\infty}\frac{t^n}{n!}E^* E^n \xi=-\sum_{n=0}^{+\infty}\frac{t^n}{n!}E^{n+1}\xi.$$
Since $t_0>0,$ we get $(iii)$.

$(iii)\implies (iv)$. This follows from an easy induction.

$(iv)\implies (v)$. We write
\begin{align*}
\exp(t_0 E^*)\exp(t_0 E)\xi&=\sum_{m,n\geq 0}\frac{t_0^{m+n}}{m!n!}(E^*)^m E^n\xi\\
&=\sum_{m,n\geq 0}\frac{t_0^{m+n}}{m!n!}(-1)^m E^{m+n}\xi\\
&=\exp(-t_0E)\exp(t_0E)\xi=\xi.
\end{align*}

$(v)\implies (i).$ We write
$$
|\exp(t_0E)\xi|^2=\langle \exp(t_0E^{*})\exp(t_0E)\xi,\xi\rangle=|\xi|^2.
$$
When $E$ is normal, a straightforward induction shows that $(iii)$ is equivalent to $(vi)$.
\end{proof}

\begin{theorem}\label{thm:continuoussemigroup}
Let $(\varphi_t)_{t\geq 0}$ be a family of holomorphic self-maps of $\CC^d$. Then $(\varphi_t)_{t\geq 0}$ is a continuous semigroup of maps belonging to $\mathcal S_c(\CC^d)$ if and only if there exist $E\in\mathcal M_d(\CC)$ and $F\in \CC^d$, such that
\begin{enumerate}[(a)]
\item for all $t\geq 0,$ $\varphi_t(z)=\exp(tE)z+(\exp(tE)-I_d)F.$
\item $\sigma(E+E^*)\subset (-\infty,0].$
\item $F\perp E\left(\bigcap_{n\geq 0}\ker( (E+E^*)E^n)\right).$
\end{enumerate}
Moreover, if these conditions are satisfied, the infinitesimal generator of $(\varphi_t)_{t\geq 0}$ is the map $G:\CC^d\to\CC^d$, $z\mapsto E(z+F).$
\end{theorem}

\begin{proof}
Assume first that $(\varphi_t)_{t\geq 0}$ is a continuous semigroup of maps of $\mathcal S_c(\CC^d)$ and write $\varphi_t(z)=A_t z+B_t$ for all $t\geq 0.$ The semigroup property ensures  that, for all $s,t\geq 0,$ $\varphi_{s+t}=\varphi_s\circ \varphi_t=\varphi_t\circ\varphi_s,$ namely
\begin{eqnarray}
A_{t+s}&=&A_tA_s=A_sA_t \nonumber \\
B_{t+s}&=&A_tB_s+B_t=A_sB_t+B_s. \label{eq:sg1}
\end{eqnarray}
Moreover, since $\varphi_t(z)\to z$ uniformly on compact subsets of $\CC^d,$ $A_t\to I_d$ as $t\to 0.$ Therefore, by \cite[Theorem I.2.8]{EN06}, there exists $E\in\mathcal M_d(\mathbb C)$ such that $A_{t}=\exp(tE)$ for all $t\geq0$.
By \cite[Lemma 2.3]{BT}, for all $t\geq 0,$ there exists $F_t\in\CC^d$ such that
$$B_t=(\exp(tE)-I_d)F_t.$$
We claim that we may choose the family $(F_t)$ such that there exists $F\in\CC^d$ with $F_t=F$ for all $t\geq 0.$
Let $E=PSP^*$ be a Schur decomposition of $E,$ namely $P$ is unitary and $S$ is an upper-triangular matrix. Denote by $\lambda_1,\dots,\lambda_d$ its diagonal coefficients. Then for $t>0,$ $\exp(tE)=P\exp(tS)P^*$ is a Schur decomposition of $\exp(tE)$ which has norm less than or equal to $1.$ Therefore the diagonal entries, $e^{t\lambda_i}$, have modulus less than or equal to $1.$ We may assume that in the first $p$ entries we find all the coefficients having modulus strictly smaller than $1.$ Then since $\|\exp(tE)\|\leq 1,$ arguing as in the proof of \cite[Proposition 2.2]{BT}, we find that
$$\exp(tE)=P\begin{pmatrix}T_t&0\\0&U_t\end{pmatrix}P^*$$
where $T_t$ is an upper-triangular matrix whose diagonal coefficients have modulus strictly less than $1$
and $U_t$ is a unitary diagonal matrix. Since any $\xi$ in $P\vect(e_{p+1},\dots,e_d)$ satisfies {$|\exp(tE)\xi|=|\xi|,$}
we find that $B_t$ is orthogonal to $P\vect(e_{p+1},\dots,e_d)$ (this last subspace does not depend on $t$).
Therefore, if we denote by $F'_t$ the orthogonal projection of $F_t$ onto $P\vect(e_{1},\dots,e_p)$, we still have
$$B_t=(\exp(tE)-I_d)F'_t.$$
Now, \eqref{eq:sg1} ensures that, for all $s,t>0,$
$$(\exp(sE)-I_d)(\exp(tE)-I_d)F'_t=(\exp(sE)-I_d)(\exp(tE)-I_d)F'_s.$$
Since $\exp(sE)-I_d$ and $\exp(tE)-I_d$ are invertible on $P\vect(e_{1},\dots,e_p)$, we get $F'_t=F'_s$ which shows the claim.

By Theorem \ref{bound-com},  $(\exp(tE))_{t\geq 0}$ is a contractive strongly continuous semigroup acting on $\CC^d.$ Therefore the Lumer-Phillips theorem implies that $\sigma(E+E^*)\subset (-\infty,0]$.

Let finally $\xi\in\CC^d$ belonging to $\bigcap_{n\geq 0}\ker((E+E^*)E^n),$ namely $E^*E^n\xi=- E^{n+1}\xi$ for all $n\geq 0$. Then for all $t\geq 0,$
{$|\exp(tE)\xi|=|\xi|$ so that,} by Theorem \ref{bound-com},
$$\langle \exp(tE)\xi,(\exp(tE)-I_d)F\rangle=0.$$
Taking the first term of the development, we get $\langle \xi,EF\rangle=0$ which yields $\langle F,E\xi\rangle=0$ since $E^*\xi=-E\xi.$

Conversely, if $(\varphi_t)_{t\geq 0}$ is a family of holomorphic self-maps of $\CC^d$ satisfying (a), (b) and (c), then
$(\varphi_t)_{t\geq 0}$ is clearly a continuous semigroup. We just need to verify that each $\varphi_t,$ $t\geq 0,$
belongs to $\mathcal S_c(\CC^d)$. But $\|\exp(tE)\|\leq 1$ by the Lumer-Phillips theorem and if $\xi\in\CC^d$ is such that {$|\exp(tE)\xi|=|\xi|,$} then
\begin{align*}
\langle \exp(tE)\xi,(\exp(tE)-I_d)F\rangle&=\langle \exp(tE^*)\exp(tE)\xi,F\rangle-\langle \exp(tE)\xi,F\rangle\\
&=\langle \xi,F\rangle-\langle \exp(tE)\xi,F\rangle\ \textrm{ by Lemma \ref{lem:sg2}}.
\end{align*}
Now observe that for $k\geq 1,$ $E^k\xi$ belongs to $E\bigcap_{n\geq 0}\ker((E+E^*)E^n)$. Therefore,
$\langle \exp(tE)\xi,F\rangle=\langle \xi,F\rangle$ and
$$\langle \exp(tE)\xi,(\exp(tE)-I_d)F\rangle=0$$
which shows that Condition (a) of Theorem \ref{bound-com} is satisfied.

Finally, the infinitesimal generator of $(\varphi_t)_{t\geq 0}$ is the map $G:\CC^d\to\CC^d$ such that, for all $z\in\CC^d,$
\begin{align*}
G(z)&=\lim_{t\to 0^+}\frac{\varphi_t(z)-\varphi_0(z)}{t}\\
&=\lim_{t\to 0^+}\frac{(\exp(tE)-I_d)(z+F)}{t}\\
&=E(z+F).
\end{align*}
\end{proof}

\begin{remark}
Using Lemma \ref{lem:sg2}, condition (c) admits several equivalent formulations. We choose the one that does not involve the variable $t.$
\end{remark}

The one-variable version of this result reads as follows.

\begin{corollary}
Let $(\varphi_t)_{t\geq 0}$ be a family of holomorphic self-maps of $\CC.$
Then $(\varphi_t)_{t\geq 0}$ is a continuous semigroup of maps belonging to $\mathcal S_c(\CC)$ if and only if there exist $a\in\mathbb C_-$ and $b\in\mathbb C$ with $b=0$ if $a\in i\mathbb R$ such that, for all $t\geq 0,$
$$\varphi_t(z)=\exp(ta)z+(\exp(ta)-1)b.$$
\end{corollary}

\begin{remark}
If $(\varphi_t)_{t\geq 0}$ writes $\varphi_t(z)=\exp(tE)z+(\exp(tE)-I_d)F$ as well as $\varphi_t(z)=\exp(tE')z+(\exp(tE')-I_d)F',$ then it is straightforward to check that $E=E'$ and $F-F'\in\ker(E).$ In particular, if $\xi\in \bigcap_{n\geq 0}\ker((E+E^*)E^n)$, then
$$\langle F',E\xi\rangle=-\langle F',E^*\xi\rangle=-\langle EF',\xi\rangle=-\langle EF,\xi\rangle=\langle F,E\xi\rangle.$$
Hence, $(E,F)$ satisfies (b) and (c) if and only if $(E',F')$ satisfies (b) and (c).
\end{remark}

\section{Strongly continuous semigroups}

We now turn to strongly continuous semigroups of composition operators on $\mathcal F(\CC^d).$

\begin{definition}
Let $\mathcal{X}$ be a Banach space and $(T_{t})_{t\geq0}$ a family of bounded operators from $\mathcal{X}$ into itself. We say that $(T_{t})_{t\geq0}$ is a \emph{semigroup} if it satisfies the following two algebraic properties:
\begin{itemize}
  \item[(i)] $T_{0}={\mathrm{Id}}_{\mathcal{X}}$, the identity map on $\mathcal{X}$;
  \item[(ii)] For every $t,s\geq0$, $T_{t}\circ T_{s}=T_{t+s}$;
\end{itemize}
If, in addition, it satisfies that
\begin{itemize}
  \item[(iii)] $\lim_{t\rightarrow0^{+}}T_{t}f=f$ for all $f\in \mathcal{X}$,
\end{itemize}
we say that it is a \emph{strongly continuous semigroup} (also known as $C_{0}$-semigroup).
\end{definition}

Moreover, the \emph{infinitesimal generator} of an operator semigroup
$\left(T_t\right)_{t \geq 0}$ on a Banach space $\mathcal{X}$ is defined as
\begin{equation}\label{t.1b1}
\Delta (f)=\lim _{t \rightarrow 0^{+}} \frac{T_t f-f}{t},
\end{equation}
where the convergence is  in the norm topology.
The set of all $f \in \mathcal{X}$ such that the limit in (\ref{t.1b1}) exists is denoted by $\mathcal{D}(\Delta)$.
If $(T_t)_{t\geq 0}$ is strongly continuous, then $\Delta$ is a closed linear operator and $\mathcal{D}(\Delta)$ is dense in $\mathcal{X}$.

\medskip

We characterize when a family $(\varphi_t)_{t\geq 0}$ generates a $C_0$-semigroup $(C_{\varphi_t})_{t\geq 0}$ on $\mathcal F(\CC^d)$. We first need the following lemma.

\begin{lemma}\label{lem:uniformbound}
Let $E\in\mathcal M_d(\CC)$ and $F\in\CC^d$ and define, for $t\geq 0,$ $\varphi_t(z)=\exp(tE)z+(\exp(tE)-I_d)F\in\mathcal S_c(\CC^d).$ Let $t_0>0.$ Then
$$\sup_{t\in[0,t_0]}\|C_{\varphi_t}\|<\infty\iff \sup_{\substack{w\in\CC^d\\ t\in[0,t_0]}}(|\varphi_t(w)|^2-|w|^2)<\infty.$$
\end{lemma}
\begin{proof}
For $w\in\CC^d,$ let $k_w$ be the reproducing kernel at $w,$ which satisfies $C_\varphi^*(k_w)=k_{\varphi(w)}$ and
$\|k_w\|=\exp(|w|^2/4).$ If $\sup_{t\in[0,t_0]}\|C_{\varphi_t}\|<\infty,$ then
$$\sup_{\substack{w\neq 0\\ t\in[0,t_0]}}\frac{\|C_{\varphi_t}^*(k_w)\|}{\|k_w\|}<\infty,$$
which implies
$$ \sup_{\substack{w\in\CC^d\\ t\in[0,t_0]}}(|\varphi_t(w)|^2-|w|^2)<\infty.$$
Conversely, let us assume that  $\sup_{\substack{w\in\CC^d\\ t\in[0,t_0]}}(|\varphi_t(w)|^2-|w|^2)<\infty.$
Observe that $\varphi_t$ is a bijection of $\CC^d$ and that the assumption means that there exists $C>0$ such that
$$-|\varphi_t^{-1}(z)|^2\leq -|z|^2+C$$ for all $t\in[0,t_0]$ and all $z\in\CC^d.$
Let $f\in\mathcal F(\CC^d)$ and $t\in[0,t_0]$.  Then
\begin{align*}
(2\pi)^d\|C_{\varphi_t}(f)\|^2&=\int_{\CC^d} |f\circ\varphi_t(w)|^2 e^{-|w|^2/2}dA(w)\\
&=|\det(\exp(-tE))|^2\int_{\CC^d} |f(z)|^2 e^{-|\varphi_t^{-1}(z)|^2/2}dA(z).
\end{align*}
Hence, $\sup_{t\in[0,t_0]}\|C_{\varphi_t}\|<\infty.$
\end{proof}

We also need that a strongly continuous semigroup must be generated by a continuous semigroup of symbols.

\begin{lemma}\label{lem:c0semitosemi}
Let $(C_{\varphi_t})_{t\geq 0}$ be a strongly continuous semigroup on $\mathcal F(\CC^d)$. Then $(\varphi_t)_{t\geq 0}$ is
a continuous semigroup of $\mathcal S_c(\CC^d).$
\end{lemma}
\begin{proof}
Since $z_k\in\mathcal F(\CC^d)$ for all $k=1,\dots,d,$ the semigroup property of $(C_{\varphi_t})_{t\geq 0}$
easily implies that $(\varphi_t)_{t\geq 0}$ is itself a semigroup with $\varphi_0(z)=z$ for all $z\in\CC^d.$
Moreover $C_{\varphi_t}(z_k)\to z_k$ in $\mathcal F(\CC^d)$ as $t\to 0$ which implies $(\varphi_t(z))_k\to z_k$ uniformly on compact subsets of $\CC^d.$ Therefore $\varphi_t\to \varphi_0$ uniformly on compact subsets of $\CC^d.$
\end{proof}

We are now ready to state our main theorem.

\begin{theorem}\label{thm:c0}
Let $(\varphi_t)_{t\geq 0}$ be a family of holomorphic self-maps of $\CC^d$. Then $(C_{\varphi_t})_{t\geq 0}$ is a strongly continuous semigroup on $\mathcal F(\CC^d)$ if and only if there exist $E\in\mathcal M_d(\CC)$ and $F\in \CC^d$, such that
\begin{enumerate}[(a)]
\item for all $t\geq 0,$ $\varphi_t(z)=\exp(tE)z+(\exp(tE)-I_d)F.$
\item $\sigma(E+E^*)\subset (-\infty,0].$
\item $F\perp E(\ker(E+E^*)).$
\end{enumerate}
\end{theorem}
\begin{remark}
As in Theorem \ref{thm:continuoussemigroup}, condition (c) may be rephrased as $EF\perp \ker(E+E^*)$ or as $EF\in\textrm{Ran}(E+E^*).$
\end{remark}
\begin{proof}
Assume first that $(C_{\varphi_t})_{t\geq 0}$ defines a strongly continuous semigroup on $\mathcal F(\CC^d)$. Then we already know that (a) and (b) are satisfied. Assume that (c) does not hold. Let $v\in \ker(E+E^*),$ $|v|=1,$
such that $\langle EF,v\rangle\neq 0.$ For $t>0,$ set $B_t=(\exp(tE)-I_d)F.$ For $w\in\CC^d,$
$$|\varphi_t(w)|^2-|w|^2=|\exp(tE)w|^2-|w|^2+2\Re e\langle\exp(tE)w,B_t\rangle+|B_t|^2.$$
We restrict ourselves to $w=\lambda v$ for $\lambda\in\CC,$ so that
$$|\varphi_t(\lambda v)|^2-|\lambda v|^2=(|\exp(tE)v|^2-1)|\lambda|^2+2\Re e(\lambda \langle\exp(tE)v,B_t\rangle)+|B_t|^2.$$
We first observe that $|\exp(tE)v|<1$: if $|\exp(tE)v|=1,$ then $v\in\bigcap_{n\geq 0}\ker((E+E^*)E^n)$ by Lemma \ref{lem:sg2} and in particular $Ev=-E^* v.$ Therefore,
$$\langle EF,v\rangle=-\langle F,Ev\rangle=0$$
by Theorem \ref{thm:continuoussemigroup}, a contradiction.
We therefore deduce that
$$\sup_{\lambda\in\CC} (|\varphi_t(\lambda v)|^2-|\lambda v|^2)=\frac{|\langle \exp(tE)v,B_t\rangle|^2}{1-|\exp(tE)v|^2}
+|B_t|^2,$$
the maximum being attained for
$$\lambda=\frac{\overline{\langle \exp(tE)v,B_t\rangle}}{1-|\exp(tE)v|^2}.$$
We observe that $|B_t|^2=O(1)$ as $t\to 0$ whereas
\begin{align*}
|\langle \exp(tE)v,B_t\rangle|^2&=|\langle \exp(tE)v,(\exp(tE)-I_d)F\rangle|^2\\
&=t^2 |\langle v,EF\rangle|^2+o(t^2).
\end{align*}
It remains to study $1-|\exp(tE)v|^2.$ We set $H=-(E+E^*)\geq 0$ and $\psi(t)=|\exp(tE)v|^2.$ Then
\begin{align*}
\psi'(t)&=\langle E\exp(tE)v,\exp(tE)v\rangle+\langle \exp(tE)v,E\exp(tE)v\rangle\\
&=-\langle H\exp(tE)v,\exp(tE)v\rangle\\
&=-|H^{1/2}\exp(tE)v|^2
\end{align*}
so that
$$1-|\exp(tE)v|^2=\int_0^t |H^{1/2} \exp(sE)v|^2ds.$$
Now,
\begin{align*}
\exp(sE)v&=v+sEv+O(s^2)\\
H^{1/2}\exp(sE)v&=sH^{1/2}Ev+O(s^2)
\end{align*}
since $v\in \ker(H).$ Integrating we get
$$1-|\exp(tE)v|^2=O(t^3).$$
We deduce that
$$ \sup_{\substack{w\in\CC^d\\ t\in[0,1]}}(|\varphi_t(w)|^2-|w|^2)=+\infty,$$
hence $\sup_{t\in[0,1]}\|C_{\varphi_t}\|=+\infty.$ This contradicts that $(C_{\varphi_t})$ is a strongly
continuous semigroup on $\mathcal F(\CC^d)$ (see \cite[Prop I.1.4.]{EN06}).

Conversely, assume that $\varphi_t(z)=\exp(tE)z+(\exp(tE)-I_d)F$ with $E$ and $F$ satisfying (b) and (c). We observe that, by Theorem \ref{thm:continuoussemigroup} (condition (c) of that theorem being implied by condition (c) of Theorem \ref{thm:c0}), each $\varphi_t$ belongs to $\mathcal S_c(\CC^d)$. We first show that $\sup_{t\in[0,1]}\|C_{\varphi_t}\|^2<\infty.$ Let $u\in\CC^d$ be such that $EF=-(E+E^*)u.$ We write
$$\varphi_t(w)=\exp(tE)(w+F)-F$$
so that, for a fixed $w\in\CC^d,$
\begin{align*}
\frac{d}{dt}\varphi_t(w)&=E\exp(tE)(w+F)\\
&=E(\varphi_t(w)+F).
\end{align*}
We deduce that
\begin{align*}
\frac{d}{dt}|\varphi_t(w)|^2&=\langle E(\varphi_t(w)+F),\varphi_t(w)\rangle+\langle \varphi_t(w),E(\varphi_t(w)+F)\rangle\\
&=-\langle H\varphi_t(w),\varphi_t(w)\rangle+2\Re e\langle EF,\varphi_t(w)\rangle\\
&=-\langle H\varphi_t(w),\varphi_t(w)\rangle+2\Re e\langle Hu,\varphi_t(w)\rangle\\
&=-\langle H(\varphi_t(w)-u),\varphi_t(w)-u\rangle+\langle Hu,u\rangle\\
&\leq \langle Hu,u\rangle
\end{align*}
since $H\geq 0.$ Integrating, we find for all $t\in[0,1]$ and all $w\in\CC^d,$
$$|\varphi_t(w)|^2-|w|^2\leq \langle Hu,u\rangle.$$
Applying Lemma \ref{lem:uniformbound}, we have shown that $\sup_{t\in[0,1]}\|C_{\varphi_t}\|^2<\infty$.

To prove that $(C_{\varphi_t})_{t\geq 0}$ is strongly continuous, we just need by \cite[Theorem I.1.6]{EN06}
to prove that it is weakly continuous, namely that for all $f,g\in\fcd,$ $\langle C_{\varphi_t}f,g\rangle\to\langle f,g\rangle$ as $t\to 0.$ Since $\sup_{t\in[0,1]}\|C_{\varphi_t}\|<\infty$, we just need to do it for $g$ in a total set of $\fcd,$ for instance for all reproducing kernels. Let $w\in\CC^d.$ Then
\begin{equation*}
\langle C_{\varphi_t}f,k_w\rangle=\langle f,C_{\varphi_t}^* k_w\rangle
=f(\varphi_t(w))\xrightarrow{t\to 0^+} f(w)=\langle f,k_w\rangle.
\end{equation*}
\end{proof}

\begin{remark}
As for Theorem \ref{thm:continuoussemigroup}, if (b) and (c) are satisfied for some possible choice of $(E,F),$ then they are satisfied for all possible choices of $(E,F)$.
\end{remark}

In their classical work \cite{BP}, Berkson and Porta have shown that $(C_{\varphi_t})_{t\geq 0}$ is a strongly continuous semigroup on $H^2(\DD)$ if and only if $(\varphi_t)_{t\geq 0}$ is a continuous semigroup of holomorphic maps in the disc. We show that this does not extend to $\mathcal F(\CC^d)$ when $d\geq 2.$

\begin{example}\label{example}
Let
$$E={\begin{pmatrix} 0&i\\i&-1/2\end{pmatrix}} \textrm{ and } F=e_2.$$
For $t\geq 0,$ let $\varphi_t$ be defined {on $\CC^2$ by $\varphi_t(z)=\exp(tE)z+(\exp(tE)-I_2)F.$}
Then $(\varphi_t)_{t\geq 0}$ is a continuous semigroup of {$\mathcal S_c(\CC^2)$} whereas $(C_{\varphi_t})_{t\geq 0}$ is not strongly continuous {on $\mathcal F(\CC^2)$.}
\end{example}
\begin{proof}
We get immediately
$$E+E^*=\begin{pmatrix}0&0\\0&-1\end{pmatrix}$$
so that $\sigma(E+E^*)=\{-1,0\}\subset (-\infty,0]$ and $\ker(E+E^*)=\mathbb Ce_1.$
Since $E e_1=ie_2\notin \mathbb Ce_1,$ the space $\bigcap_{n\geq 0}\ker((E+E^*)E^n)$ is equal to $\{0\},$
so that  assumption (c) of Theorem \ref{thm:continuoussemigroup} is satisfied.

On the contrary, we observe that $e_1\in\ker(E+E^*)$ and that $Ee_1=ie_2$
is not orthogonal to $F.$
This prevents $(C_{\varphi_t})_{t\geq 0}$ from being a strongly continuous semigroup {on $\mathcal F(\CC^2).$}
\end{proof}

%

Observe that the previous matrix $E$ is diagonalizable, but that it is not normal. When the matrix $E$ is normal, such a phenomenon cannot arise.

\begin{theorem}\label{thm:c0normal}
Let $E\in\mathcal M_d(\CC)$ be normal and let $F\in\CC^d.$ For $t\geq 0,$ define $\varphi_t$ by
$$\varphi_t(z)=\exp(tE)z+(\exp(tE)-I_d)F,\ z\in\CC^d.$$
Then the following assertions are equivalent:
\begin{enumerate}[(a)]
\item $\exists t_0>0,$ $\varphi_{t}\in\mathcal S_c(\CC^d)$ for all $t\in[0,t_0]$.
\item $(\varphi_t)_{t\geq 0}$ is a continuous semigroup of $\mathcal S_c(\CC^d)$.
\item $(C_{\varphi_t})_{t\geq 0}$ is a strongly continuous semigroup on $\mathcal F(\CC^d)$.
\item  $\sigma(E+E^*)\subset (-\infty,0]$ and $F\perp E(\ker(E+E^*)).$
\end{enumerate}
\end{theorem}

\begin{proof}
Since $E$ is normal, there exists a unitary matrix $P\in\mathcal M_d(\CC)$ such that $P^*EP=D$ is a diagonal matrix. Defining for $t\geq 0$ and $z\in\CC^d$, $\psi_t(z)=\exp(tD)z+(\exp(tD)-I_d)P^* F,$ \eqref{eq:unitary}
tells us that $C_{\varphi_t}$ is unitarily similar to $C_{\psi_t}$. Hence, we may assume that $E$ is diagonal.

Using the previous results, we know that $(d)\implies (c)\implies(b)\implies (a)$. Let us prove that $(a)\implies (d).$
We may assume that $t_0=1$ so that $\varphi_1\in\mathcal S_c(\CC^d)$. Let us write $E=\textrm{diag}(\lambda_1,\dots,\lambda_d)$. Since $\|\exp(E)\|\leq 1,$ then $\Re e(\lambda_i)\leq 0$ for $i=1,\dots,d.$ Reordering the eigenvalues if necessary, we may consider $p\leq q\in\{0,\dots,d\},$ {such that $\Re e(\lambda_i)<0$ for $i=1,\dots,p$, $\Re e(\lambda_i)=0$ for $i=p+1,\dots,d$} and $\lambda_i=0$ if and only if $i=p+1,\dots,q.$
This yields $\sigma(E+E^*)\subset (-\infty,0]$ and
\begin{equation*}
\ker(E+E^*)=\textrm{span}(e_{p+1},\dots,e_d)
{=\left\{\xi\in\CC^d:\ |\exp(E)\xi|=|\xi|\right\}.}
\end{equation*}
By Theorem \ref{bound-com}, $(\exp(tE)-I_d)F\in \textrm{span}(e_1,\dots,e_p)$ for all $t\in[0,1],$ so that, choosing $t$
with $\exp(t\lambda_i)\neq 1 $ for $i>q$,
$$F\in\textrm{span}(e_1,\dots,e_q)=\textrm{span}(e_{q+1},\dots,e_d)^\perp=(E(\ker(E+E^*)))^\perp.$$
\end{proof}

\begin{remark}
Of course this result holds for all semigroups $(\varphi_t)_{t\geq 0}$ if $d=1.$ This improves \cite[Theorem 2.3]{ChP} where the implication $(d) \implies (c)$ was not considered.
\end{remark}

\begin{corollary}
Let $\varphi(z)=Az+B\in\mathcal S_c(\CC^d)$ with $A$ invertible and normal. Then there exists a family $(\varphi_t)_{t\geq 0}$ of $\mathcal S_c(\CC^d)$ such that $(C_{\varphi_t})_{t\geq 0}$ is strongly continuous and $\varphi_1=\varphi.$
\end{corollary}

\begin{proof}
We may assume that $A$ is diagonal, $A=\textrm{diag}(\lambda_1,\dots,\lambda_d)$. {Since $\sigma(A)\subset \overline{\DD}$,} reordering the eigenvalues if necessary,
{we may define $p\leq q\in\{0,\dots,d\}$ such that $|\lambda_k|<1$ if and only if $k\leq p$ and $\lambda_k=1$ if and only if $k=p+1,\dots,q$.} By Theorem \ref{bound-com}, we know that $B\in\textrm{span}(e_1,\dots,e_p)$.
{We set $E=\textrm{diag}(\log(\lambda_1),\dots,\log(\lambda_p),0,\dots,0,i\theta_{q+1},\dots,i\theta_d)$,} where $\log(\lambda_k)$ is any (complex) logarithm of $\lambda_k$ and $\lambda_j=e^{i\theta_j}$ for $j\geq q+1$. Then one may find $F\in\textrm{span}(e_1,\dots,e_p)$ such that $B=(\exp(E)-I_d)F$ and we may observe that $F\perp (E(\ker(E+E^*))).$ We conclude by defining $\varphi_t$ on $\CC^d$ for $t\geq 0$ by
$$\varphi_t(z)=\exp(tE)z+(\exp(tE)-I_d)F,\ z\in\CC^d,$$
and by applying the previous theorem.
\end{proof}

We end this section by computing the infinitesimal generator of the semigroups above and its spectrum.
Throughout, for a holomorphic
$f:\mathbb{C}^d\to\mathbb{C}$ and $G=(G_1,\dots,G_d):\mathbb{C}^d\to\mathbb{C}^d$,
we write
\begin{equation*}
\langle \nabla f(z),G(z)\rangle := \sum_{j=1}^{d}\partial_j f(z)\,G_j(z),
\end{equation*}
the bilinear pairing (not the hermitian one).

\medskip

We shall need two elementary lemmas. The first one is folklore; we include
its short proof for completeness. Recall that if $(H_n)_{n\ge 1}$ are Hilbert
spaces and $A_n$ is a closed operator on $H_n$, the direct sum
$A=\bigoplus_n A_n$ on $H=\bigoplus_n H_n$ is defined on the maximal domain
$D(A)=\{u=(u_n)\in H:\ u_n\in D(A_n),\ (A_nu_n)\in H\}$.

\begin{lemma}\label{lem:directsum}
For $n\ge 1$, let $(S^n_t)_{t\ge 0}$ be a strongly continuous semigroup on a
Hilbert space $H_n$, with generator $A_n$, and assume that
$\|S^n_t\|\le Me^{\omega t}$ for all $n\ge 1$ and all $t\ge 0$, with $M,\omega$
independent of $n$. Set $H=\bigoplus_{n} H_n$ and $S_t=\bigoplus_n S^n_t$.
Then $(S_t)_{t\ge0}$ is a strongly continuous semigroup on $H$ whose generator
is $A=\bigoplus_n A_n$. Moreover, for every $\mu\in\mathbb{C}$,
\begin{equation*}
  \mu\in\rho(A)\iff {\mu\in \bigcap_{n\geq1}\rho(A_n)}
  \text{ and }\ \sup_{n\ge1}\|R(\mu,A_n)\|<\infty,
\end{equation*}
and in that case $R(\mu,A)=\bigoplus_n R(\mu,A_n)$.
\end{lemma}

\begin{proof}
Since $\|S_t u\|^2=\sum_n\|S^n_tu_n\|^2\le M^2e^{2\omega t}\|u\|^2$, the family
$(S_t)$ is uniformly bounded on compact time intervals; together with the
strong continuity of each $(S^n_t)$, this gives the strong continuity of $(S_t)$ on
the dense subspace of finitely supported vectors, hence on $H$. Let $A'$
denote its generator.

If $u\in D(A')$, then applying the orthogonal projection $P_n$ onto $H_n$,
{which is bounded and commutes with $S_t$,} gives
$t^{-1}(S^n_tu_n-u_n)\to (A'u)_n$; hence $u_n\in D(A_n)$, {$A_nu_n=(A'u)_n$ and
$(A_nu_n)\in H$, i.e. $u\in D(A)$ and $Au=A'u$. Thus $D(A')\subset D(A)$.}

Fix now $\lambda$ with $\Re e(\lambda)>\omega$. Then $\lambda\in\rho(A')$, so
$\lambda-A'$ maps $D(A')$ onto $H$. Moreover $\lambda-A$ is injective: if
$Au=\lambda u$ then $A_nu_n=\lambda u_n$ for every $n$, while
$\lambda\in\rho(A_n)$ because $\|S^n_t\|\le Me^{\omega t}$; hence $u=0$. Given
$u\in D(A)$, pick $w\in D(A')$ with $(\lambda-A')w=(\lambda-A)u$; {since
$D(A')\subset D(A)$} this forces $u=w\in D(A')$. Therefore $A=A'$.

Assume now $\mu\in\rho(A_n)$ for every $n$ and
$M_0:=\sup_n\|R(\mu,A_n)\|<\infty$, and put $R:=\bigoplus_nR(\mu,A_n)$, a
bounded operator of norm at most $M_0$. Given $f=(f_n)\in H$, the vectors
$u_n:=R(\mu,A_n)f_n$ satisfy $\sum_n\|u_n\|^2\le M_0^2\|f\|^2$ and
$A_nu_n=\mu u_n-f_n$, whence
\[
  \sum_n\|A_nu_n\|^2\le 2|\mu|^2M_0^2\|f\|^2+2\|f\|^2<\infty .
\]
Thus $Rf\in D(A)$, and one checks at once that $(\mu-A)R=I$ on $H$ and
$R(\mu-A)=I$ on $D(A)$, so that $\mu\in\rho(A)$.

Conversely, each $H_n$ reduces $A$, so $R(\mu,A)$ commutes with the orthogonal
projection onto $H_n$ and its restriction to $H_n$ is a bounded inverse of
$\mu-A_n$ of norm at most $\|R(\mu,A)\|$.
\end{proof}

\begin{lemma}\label{lem:theta}
Let $\theta_{p+1},\dots,\theta_d\in\mathbb{R}$ and set
\[
  \Theta_+:=\Big\{\textstyle\sum_{k>p}\alpha_k\theta_k:\ \alpha_k\in\ZZ_+\Big\},
  \qquad
  \Theta:=\textstyle\sum_{k>p}\theta_k\mathbb{Z}.
\]
Then $\overline{\Theta_+}=\mathbb{R}$ if and only if $\Theta$ is dense in 
$\mathbb{R}$ and there exist $j\ne k$ with $\theta_j\theta_k<0$. Otherwise
$\Theta_+$ is closed and locally finite.
\end{lemma}

\begin{proof}
If all the $\theta_k$ have the same sign, every bounded subset of $\mathbb{R}$
contains only finitely many elements of $\Theta_+$, so $\Theta_+$ is locally
finite. If $\Theta\subset c\mathbb{Z}$ for some $c>0$ the same conclusion
holds. Conversely, assume $\Theta$ dense and there exists $j\neq k$ with $\theta_j\theta_k<0$.
Fix $b\in\Theta_+,$ $b<0$.
Assume that there exists $a\in\Theta_+$, $a>0$, such that
$a/|b|\notin\mathbb{Q}.$ Then $\{ma+nb:m,n\ge1\}$ is already dense in
$\mathbb{R}$ and $\overline{\Theta_+}=\mathbb R.$ Otherwise, for all $a\in\Theta_+$, $a>0,$ $a/|b|=q/q'$ with $q,q'\in\mathbb{N}$. Hence $q'a+qb=0$ and
$(q'-1)a+qb=-a\in \Theta_+$. Therefore (we can exchange the role played by the positive and the negative elements of $\Theta_+$),  $\overline{\Theta_+}$ is a closed subgroup of $\mathbb{R}$
containing every $\theta_k$; hence it contains $\overline{\Theta}=\mathbb{R}$.
\end{proof}

\medskip

\begin{theorem}\label{thm:generator}
Let $(C_{\varphi_t})_{t\ge0}$ be a strongly continuous semigroup on
$\mathcal F(\mathbb{C}^d)$. Let $E\in \mathcal M_d(\mathbb{C})$ and $F\in\mathbb{C}^d$ be such
that $\varphi_t(z)=\exp(tE)z+(\exp(tE)-I_d)F$ for all $t\ge0$ and all
$z\in\mathbb{C}^d$. Then the infinitesimal generator of $(C_{\varphi_t})_{t\ge0}$
is
\[
  \Delta(f)(z)=\langle\nabla f(z),G(z)\rangle,\qquad z\in\mathbb{C}^d,
\]
where $G(z)=E(z+F)$.

Moreover, let $\lambda_1,\dots,\lambda_d$ be the eigenvalues of $E$, repeated
according to algebraic multiplicity and ordered so that $\Re e(\lambda_k)<0$ for
$k=1,\dots,p$ and $\lambda_k=i\theta_k$, $\theta_k\in\mathbb{R}$, for
$k=p+1,\dots,d$. Set
\[
  S:=\left\{\sum_{k=1}^{d}\alpha_k\lambda_k:\ \alpha\in\ZZ_+^d\right\}.
\]
Then $\sigma(\Delta)=\overline{S}$. More precisely:
\begin{enumerate}
\item[(a)] if $\theta_{p+1}\mathbb{Z}+\dots+\theta_d\mathbb{Z}$ is dense in
  $\mathbb{R}$ and there exist $j\ne k$ with $\theta_j\theta_k<0$, then
  \[
    \sigma(\Delta)=\bigcup_{\alpha\in\ZZ_+^p}
      \Big(\sum_{k=1}^{p}\alpha_k\lambda_k+i\mathbb{R}\Big);
  \]
\item[(b)] otherwise $S$ is closed and $\sigma(\Delta)=S$.
\end{enumerate}
\end{theorem}

\begin{proof}
By the chain rule, for $f\in\mathcal F(\mathbb{C}^d)$ in the domain of $\Delta$,
\[
  \Delta(f)(z)=\lim_{t\to0^+}\frac{f(\varphi_t(z))-f(z)}{t}
             =\langle\nabla f(z),E(z+F)\rangle .
\]

For $k=1,\dots,d$, since $\lambda_k$ is an eigenvalue of $E$, there exists a
nonzero linear form $L_k\in(\CC^d)^*$ such that $L_k\circ E=\lambda_k L_k$. Then
$L_k\circ\exp(tE)=e^{t\lambda_k}L_k$ for all $t\geq 0$. Setting
$f_k(z)=L_k(z+F)$, we get, for all $t\geq 0$ and all $z\in\CC^d$,
\[
  C_{\varphi_t}f_k(z)=L_k\big(\exp(tE)z+(\exp(tE)-I_d)F+F\big)
  =L_k\big(\exp(tE)(z+F)\big)=e^{t\lambda_k}f_k(z).
\]
Now let $\alpha\in\ZZ_+^d$ and set $f_\alpha=\prod_{k=1}^d f_k^{\alpha_k}$,
which is a nonzero polynomial, hence belongs to $\fcd$. Since
$C_{\varphi_t}$ is multiplicative on polynomials, we get
$C_{\varphi_t}f_\alpha=e^{ts_\alpha}f_\alpha$ with
$s_\alpha=\sum_{k=1}^d\alpha_k\lambda_k$. Therefore
\[
  \frac{C_{\varphi_t}f_\alpha-f_\alpha}{t}
  =\frac{e^{ts_\alpha}-1}{t}\,f_\alpha\xrightarrow[t\to0^+]{} s_\alpha f_\alpha
\]
in $\fcd$, so that $f_\alpha\in D(\Delta)$ and $\Delta f_\alpha=s_\alpha f_\alpha$.
Hence $S\subset\sigma_p(\Delta)\subset\sigma(\Delta)$. Since  $\Delta$ is a closed operator, it follows that $\overline{S}\subset\sigma(\Delta)$.

To prove the converse inclusion we will use a diagonalization of our operator.
We do a normalization of the symbol.
Arguing as in the proof of Theorem \ref{thm:continuoussemigroup}, there is a unitary
$P\in \mathcal M_d(\mathbb{C})$, an upper triangular matrix $E_1\in\mathcal M_p(\mathbb C)$
and a diagonal matrix $ E_2=\operatorname{diag}(i\theta_{p+1},\dots,i\theta_d)$ such that
$$E=P\begin{pmatrix}E_1&0\\0&E_2\end{pmatrix}P^*.$$ Hence, for all $t\geq 0,$
$$  \exp(tE)=P\begin{pmatrix}T_t&0\\0&U_t\end{pmatrix}P^*,$$
where $T_t\in \mathcal M_p(\mathbb{C})$ is upper triangular with diagonal entries of
modulus strictly less than $1$ and $U_t\in \mathcal M_{d-p}(\mathbb{C})$ is a unitary
diagonal matrix; moreover we may choose $F\in P(\mathbb{C}^p\times\{0\}^{d-p})$.
 By \eqref{eq:unitary},
$C_{\varphi_t}$ is unitarily equivalent to the semigroup associated with
$P^*EP$ and $P^*F$; since neither $\sigma(\Delta)$ nor $S$ is affected, we may
and do assume $P=I_d$, so that $E=E_1\oplus E_2$ and $F=(F_1,0)$ with
$F_1\in\mathbb{C}^p$.

Write $z=(z',z'')\in\mathbb{C}^p\times\mathbb{C}^{d-p}$ and
\begin{equation*}
\varphi^1_t(z')=\exp(tE_1)z'+(\exp(tE_1)-I_p)F_1,
\end{equation*}
so that $\varphi_t(z)=(\varphi^1_t(z'),\exp(tE_2)z'')$. For
$\beta\in\ZZ_+^{d-p}$ let
\begin{equation*}
\mathcal{H}_\beta:=\big\{g(z')\,(z'')^\beta:\ g\in\mathcal F(\mathbb{C}^p)\big\}.
\end{equation*}
By Fubini's theorem,
$\mathcal F(\mathbb{C}^d)=\bigoplus_{\beta}\mathcal{H}_\beta$ is an orthogonal
decomposition, and $\mathcal{H}_\beta$ is isometrically isomorphic to
$\mathcal F(\mathbb{C}^p)$ after normalising $(z'')^\beta$. Since $E_2$ is diagonal,
$\big(\exp(tE_2)z''\big)^\beta=e^{itc_\beta}(z'')^\beta$ with
\begin{equation*}
  c_\beta:=\sum_{k>p}\theta_k\beta_k\in\Theta_+,
\end{equation*}
so that, writing $T_t:=C_{\varphi^1_t}$ on $\mathcal F(\mathbb{C}^p)$,
\begin{equation*}\label{eq:decomp}
  C_{\varphi_t}\big(g\otimes(z'')^\beta\big)
    =e^{itc_\beta}\,(T_tg)\otimes(z'')^\beta .
\end{equation*}
In particular every $\mathcal{H}_\beta$ reduces $C_{\varphi_t}$, and
$(T_t)_{t\ge0}$ is a strongly continuous semigroup on $\mathcal F(\mathbb{C}^p)$ (it is
the restriction of $(C_{\varphi_t})$ to $\mathcal{H}_0$). Denoting by
$\Delta_1$ its generator and applying Lemma \ref{lem:directsum} to the
uniformly bounded family $(e^{itc_\beta}T_t)_\beta$, we obtain
\begin{equation}\label{eq:gensum}
  \Delta=\bigoplus_{\beta\in\ZZ_+^{d-p}}(\Delta_1+ic_\beta).
\end{equation}
An important point for the following is that, for all $t> 0,$ $T_t$ is compact. Indeed,
by Theorem \ref{bound-com},
we just need to verify that $\|\exp(tE_1)\|<1$. We already know that $\|\exp(tE)\|\le1$,
hence $\|\exp(tE_1)\|\le1$. If $\|\exp(tE_1)\|=1$, then Lemma \ref{lem:sg2}
ensures that $\|\exp(sE_1)\|=1$ for all $s>0$, which contradicts  $\exp(sE_1)\to0$
as $s\to\infty,$ a consequence of $\sigma(E_1)\subset\{z:\ \Re e(z)<0\}$.

\medskip

We shall compute the spectrum of $\Delta$ using Lemma \ref{lem:directsum}.
Thus we first compute the spectrum of $\Delta_1$.

Let
$$S_1=\left\{\sum_{k=1}^p \alpha_k\lambda_k:\ \alpha\in\ZZ_+^p\right\}.$$
Observe that $S_1$ is closed since $\Re e(\lambda_k)<0$ for $k=1,\dots,p$ (this is why it is easier to work
with $E_1$) and we already know (by the first part of the proof applied to $(\varphi^1_t)$ and $\Delta_1$) that $S_1\subset\sigma(\Delta_1)$.
Let us now prove the reverse inclusion. By \cite[Theorem V.2.5]{EN06}, for all $t\geq 0,$
$$e^{t\sigma(\Delta_1)}\subset\sigma(C_{\varphi^1_t}).$$
On the other hand, by a result of \cite{JPZ}, for all $t\geq 0,$
$$\sigma(C_{\varphi^1_t})=\overline{\left\{\exp\left(t\sum_{k=1}^p \alpha_k\lambda_k\right):\ {\alpha\in\ZZ_+^p}\right\}}.$$
The set $\left\{\exp\left(t\sum_{k=1}^p \alpha_k\lambda_k\right):\ {\alpha\in\ZZ_+^p}\right\}\cup\{0\}$ is closed for
all $t>0.$
Let $\mu\in\sigma(\Delta_1)$. For all $t>0,$ there exists $k(t)\in\ZZ$ and $\alpha(t)\in\ZZ_+^p$ such that
$$\mu-\sum_{k=1}^p \alpha_k(t)\lambda_k=\frac{2\pi i k(t)}{t}.$$
If $k(t)\neq 0$ for all $t>0$, then $\{2\pi i k(t)/t:\ t>0\}$ is uncountable whereas
\begin{equation*}
\left\{\mu-\sum_{k=1}^p \alpha_k\lambda_k:\ \alpha\in\ZZ_+^p\right\}
\end{equation*}
is countable, a contradiction. Therefore, for at least one $t>0,$ $k(t)=0$ so that $\mu\in S_1.$

In order to apply Lemma \ref{lem:directsum}, we still need a uniform resolvent estimate.
More precisely, we claim that for every $x\in\mathbb{R}$ and every $\delta>0$,
\begin{equation}\label{eq:unifres}
  \sup\Big\{\|R(\mu,\Delta_1)\|:\ \Re e(\mu)=x,\ \operatorname{dist}(\mu,S_1)\ge\delta\}
  <\infty .
\end{equation}

Let us fix $x\in\mathbb R$ and let us introduce
$ \Sigma_x:=S_1\cap\{z\in\mathbb C:\ \Re e(z)\ge x\}$ and observe that $\Sigma_x$ is finite.
By \cite[Proposition V.1.17]{EN06} one gets a
decomposition $\mathcal F(\mathbb{C}^p)=X\oplus Y$ into $\Delta_1$-invariant closed
subspaces such that $\Delta_1|_X$ is bounded,
$\sigma(\Delta_1|_X)=\Sigma_x$ and $\sigma(\Delta_1|_Y)=S_1\setminus\Sigma_x$.

Let $P$ be the projection given by \cite[Proposition V.1.17]{EN06}, and recall that
it is obtained there as the spectral projection of the bounded operator
$R:=R(\nu,\Delta_1)$ for some $\nu\in\rho(\Delta_1)$ fixed; in particular $P$ lies in the range of the
holomorphic functional calculus of $R$, so it commutes with every bounded
operator commuting with $R$. Now by  \cite[Lemma II.1.3]{EN06} we have
$T_t D(\Delta_1)\subset D(\Delta_1)$ and $\Delta_1T_tu=T_t\Delta_1u$ for
$u\in D(\Delta_1)$, whence $T_tR=RT_t$ for every $t\ge0$. Therefore
$PT_t=T_tP$, so that $X=\operatorname{ran}P$ and $Y=\ker P$ are
$T_t$-invariant, and $(T_t|_Y)_{t\ge0}$ is a strongly continuous
semigroup with generator $\Delta_1|_Y$.

On $X$, the operator $\Delta_1|_X$ is bounded with finite spectrum $\Sigma_x$,
so $\mu\mapsto\|R(\mu,\Delta_1|_X)\|$ is continuous on
$\mathbb{C}\setminus\Sigma_x$ and tends to $0$ as $|\mu|\to\infty$; it is
therefore bounded on the closed set
$\{\Re e(\mu)=x,\ \operatorname{dist}(\mu,\Sigma_x)\ge\delta\}$.
On $Y$, each operator $T_t|_Y$, $t>0,$ being compact, the semigroup $(T_t|_Y)_{t\geq 0}$ is eventually
norm-continuous, so \cite[Corollary V.2.9]{EN06} gives equality between the spectral bound and the growth bound:
$$\omega_0:=\inf\{\omega\in\mathbb R:\ \exists M>0,\ \forall t\geq 0,\ \|T_t|_Y\|\leq Me^{\omega t}\}= \max\{\Re e(s):\ s\in S_1\setminus\Sigma_x\}.$$
Fix $x'\in(\omega_0,x)$ and $M\ge1$ such that $\|T_t|_Y\|\le Me^{x't}$ for all
$t\ge0$. Since $\Re e(\mu)=x>x'$, \cite[Theorem II.1.10.(iii)]{EN06} applied to the
semigroup $(T_t|_Y)_{t\ge0}$ yields
\[
  \|R(\mu,\Delta_1|_Y)\|\le\frac{M}{x-x'},
\]
a bound independent of $\Im m(\mu)$.
Adding the two
estimates gives \eqref{eq:unifres}.

\medskip

We are now ready to conclude.
Let $\mu\notin\overline{S}$ and put
$\delta:=\operatorname{dist}(\mu,S)>0$. For every $\beta$ we have
$S_1+ic_\beta\subset S$, whence
\begin{equation*}
\operatorname{dist}(\mu-ic_\beta,S_1)=\operatorname{dist}(\mu,S_1+ic_\beta)\ge\delta,
\end{equation*}
and $\Re e(\mu-ic_\beta)=\Re e(\mu)$ for every $\beta$. In particular
$\mu-ic_\beta\in\rho(\Delta_1)$, and \eqref{eq:unifres} applied with
$x=\Re e(\mu)$ and this $\delta$ gives
$$\sup_\beta\|R(\mu,\Delta_1+ic_\beta)\|
   =\sup_\beta\|R(\mu-ic_\beta,\Delta_1)\|<\infty.$$
Lemma \ref{lem:directsum} then yields $\mu\in\rho(\Delta)$. Therefore
$\sigma(\Delta)=\overline{S}$.

It remains to identify $\overline{S}$. We first check that
$S_1+i\overline{\Theta_+}$ is closed: if $s_n+ic_n\to\zeta$ with $s_n\in S_1$
and $c_n\in\overline{\Theta_+}$, then $\Re e( s_n)\to\Re e(\zeta)$, so $\Re e(s_n)\ge\Re e(\zeta)-1$
for $n$ large and since $\Sigma_{\Re e(\zeta)-1}$ is finite, we may assume $s_n=s$ constant after
extracting; then $c_n\to(\zeta-s)/i\in\overline{\Theta_+}$. Since
$S_1+i\overline{\Theta_+}\subset\overline{S}$ by continuity, we get
$\overline{S}=S_1+i\overline{\Theta_+}$. Lemma \ref{lem:theta} now gives
$\overline{\Theta_+}=\mathbb{R}$ in case (a), which is assertion (a), and
$\overline{\Theta_+}=\Theta_+$ in case (b), which is assertion (b).
\end{proof}

\begin{remark}
The spectral bound of $\Delta$ is
$s(\Delta)=\max\{\Re e(\nu):\nu\in\sigma(\Delta)\}=0$, attained at
$\alpha=0$, and it is equal to the growth bound of $(C_{\varphi_t})_{t\ge0}$,
\[
  \omega_0:=\inf\{\omega\in\mathbb{R}:\ \exists M>0,\
   \|C_{\varphi_t}\|\le M e^{\omega t}\ \ \forall t\ge0\}=0,
\]
since $\omega_0=\frac1t\log r(C_{\varphi_t})=0$ by \cite[Proposition V.1.22]{EN06}.
Note that this equality holds here by direct computation and not by the
general result \cite[Corollary V.2.9]{EN06}, which does not apply to
$(C_{\varphi_t})_{t\ge0}$.
\end{remark}

\section{Reducibility}\label{irr}

As usual, denote by $\mathcal{L}(H)$ the algebra of all bounded linear operators on a Hilbert space $H$. Given any $T\in \mathcal{L}(H)$, we recall that the \emph{commutant} and \emph{bicommutant} of $T$ are defined as
\begin{equation*}
  \{T\}':=\left\{A\in\mathcal{L}(H):TA=AT\right\}, \quad \{T\}''=\bigcap_{A\in\{T\}'}\{A\}'
\end{equation*}
respectively. Also let $\alg (T):=\{p(T):p\ \mbox{is a polynomial}\}.$
It is easy to see that $\overline{\alg (T)}^{\sigma}$, the closure of $\alg (T)$ in the weak operator topology $\sigma$, is a commutative algebra such that
\begin{equation*}
  \overline{\alg (T)}^{\sigma}\subset\{T\}''\subset\{T\}'.
\end{equation*}
We say that an operator $T$ has a \emph{minimal commutant} if
$\overline{\alg (T)}^{\sigma}=\{T\}'.$
Likewise, $T$ has the \emph{double commutant property} if $\overline{\alg (T)}^{\sigma}=\{T\}''$. See \cite{LLSR18,LLSR19} for a complete characterization of the minimal commutant property and double commutant property for composition operators induced by linear fractional self-maps of the unit disk $\DD$ on the Hardy space $H^2(\DD)$ and \cite{BWY} for an extension to the context of several complex variables.

Recently, the existence of non-trivial idempotents lying in the commutant or the bicommutant of composition operators on $H^{2}(\DD)$ has been considered in \cite{Ja}.
As is well-known, such an idempotent is intricately linked with invariant subspaces and the decomposition of the space, see for example \cite{Co,Ha1967}.
$T\in \mathcal{L}(H)$ is called an \emph{irreducible operator} if the commutant $\{T\}'$ does not contain any non-trivial orthogonal projection, or equivalently, there is no non-trivial reducing subspace of $T$. Note that irreducibility is only unitary-invariant and not similarity-invariant.
Furthermore, $T\in \mathcal{L}(H)$ is called \emph{strongly irreducible} if $STS^{-1}$ is irreducible for all invertible operators $S$; otherwise $T$ is strongly reducible.
Then strongly irreducible operators are irreducible and strong irreducibility is invariant under similarity.
Following from \cite[Proposition 2.25]{JW}, $T\in \mathcal{L}(H)$ is strongly irreducible if and only if
there does not exist any non-trivial idempotent $P$ lying in $\{T\}'$,
or equivalently, there do not exist two non-trivial invariant subspaces $M$ and $N$ of $T$ such that $\mathcal{H}=M\oplus N$.

The aim of this section is primarily devoted to characterizing the existence of non-trivial idempotents (that is, $J\in\mathcal{L}(\fcd)$ such that $J^{2}=J$ and $J\neq0,Id_{\fcd}$) lying in the commutant or the bicommutant of composition operators.
For this, we first recall the functional commutant (or bicommutant) property on a complex Banach space $\mathcal{X}$, given in \cite{Ja}. Given a subset $G\subset\mathbb{C}$,
denote by $\mathcal{C}(G)$ the algebra consisting of all continuous functions $f:G\rightarrow\mathbb{C}$, equipped with the pointwise sum and multiplication of functions.
An operator $T\in\mathcal{L}(\mathcal{X})$ has the \emph{functional commutant property} (respectively, \emph{functional bicommutant property}) if there exist a connected subset $G\subset\mathbb{C}$, a subalgebra $\mathcal{A}\subset\mathcal{C}(G)$ and an algebra isomorphism $\Phi:\mathcal{A}\rightarrow\{T\}'$ (respectively, $\Phi:\mathcal{A}\rightarrow\{T\}''$).
We also recall the following relations between the functional commutant (or bicommutant) property and the existence of idempotents lying in such algebras in \cite{Ja}.

\begin{lemma}\label{nece-fc}
If $T\in\mathcal{L}(\mathcal{X})$ has the functional commutant property (respectively, functional bicommutant property),
then there exist no non-trivial idempotents lying in $\{T\}'$ (respectively, $\{T\}''$).
\end{lemma}

Combining this with \cite[Proposition 2.18]{JW}, we obtain the following implication.

\begin{lemma}\label{FCP-spectra}
If $T\in\mathcal{L}(\mathcal{X})$ has the functional commutant property,
then $T$ is strongly irreducible.
\end{lemma}

As illustrated at the end of the work \cite{Ja} on $H^{2}(\mathbb{D})$, the concepts of minimal commutant or the double commutant property are not sufficient to assure the non-existence of non-trivial idempotents in the commutant or the bicommutant of an operator.
The same phenomenon occurs on the Fock space.
Precisely, we say that the complex numbers $\lambda_1,\dots,\lambda_s$ are \emph{independent} if, for any $\alpha,\beta\in\ZZ_+^s,$ $\lambda^\alpha=\lambda^\beta$ implies $\alpha=\beta.$
In particular, if $\lambda_i=\lambda_j$ for some $i\neq j$ or if $\lambda_i=0$ for some $i,$ then $\lambda_1,\dots,\lambda_s$ are not independent.
Then there exist composition operators having a minimal commutant, while they  fail to possess the functional commutant property, due to \cite[Theorem 5.2]{BWY} and the following result.

\begin{theorem}\label{spe-fc}
Let $\varphi(z)=Az+B\in \mathcal{S}_{c}(\mathbb{C}^{d})$ with $A\in \mathcal M_d(\mathbb C)$ and $B\in \mathbb{C}^{d}$.
If either $\sigma(A)\subset\mathbb{D}$ or $\sigma(A)\subset\mathbb{T}$ and the eigenvalues of $A$ are independent,
then there exist non-trivial idempotents lying in $\{C_{\varphi}\}''$.
As a consequence,
$C_{\varphi}$ possesses neither the functional  commutant nor the bicommutant property.
\end{theorem}

\begin{proof}
Let either $\sigma(A)\subset\mathbb{D}$, or $\sigma(A)\subset\mathbb{T}$ are independent. We first claim that $\ker(C_{\varphi}-I)=\mathbb{C}\mathbb{1}$ and  $\ker(C_{\varphi}^{\ast}-I)=\mathbb{C}k_w$ for some $w\in\CC^d.$ If $\sigma(A)\subset\mathbb{T}$, then $B=0$ and $C_{\varphi}$ is similar to $C_{\psi}$ with $\psi(z)=Sz$, where $S$ is the diagonal matrix containing all the eigenvalues of $A$, following from Lemma \ref{simility}.
At the same time, $C_{\psi}^{\ast}=C_{\widehat{\psi}}$ with $\widehat{\psi}(z)=S^{\ast}z$, due to \cite[Lemma 2]{CMS}. Then the desired result with $w=0$ can be obtained easily by the hypotheses that $\sigma(A)\subset\mathbb{T}$ are independent.

If $\sigma(A)\subset\mathbb{D}$, then the sequence $(C_{\varphi}^{n})$ converges in norm to the operator mapping a function $f$ to the constant function with value $f((I_{d}-PNP^{\ast})^{-1}B)$ as in Theorem \ref{thm:iteration}. Clearly, $\ker(C_{\varphi}-I)=\mathbb{C}\mathbb{1}$.
Pick now $g\in\fcd$ such that $C_{\varphi}^{\ast}g=g$. This implies that $C_{\varphi^{[n]}}^{\ast}g=g$ for each $n\in \mathbb{N}$.
Thus, for each $f\in \fcd$,
\begin{equation*}
  \langle g,f\rangle=\langle C_{\varphi^{[n]}}^{\ast}g,f\rangle=\langle g,C_{\varphi^{[n]}}f\rangle,
\end{equation*}
and meanwhile,
\begin{equation*}
\langle g,C_{\varphi^{[n]}}f\rangle\rightarrow \langle g,f((I_{d}-PNP^{\ast})^{-1}B)\mathbb{1}\rangle
=\overline{f((I_{d}-PNP^{\ast})^{-1}B)}g(0)
\end{equation*}
as $n\rightarrow\infty$.
Therefore, $g=g(0)k_w$ with $w=(I_{d}-PNP^{\ast})^{-1}B.$

Furthermore, considering the rank-one operator $J=\mathbb{1}\otimes k_w$, that acts as
\begin{equation*}
  (\mathbb{1}\otimes k_w)f=f(w)\mathbb{1}, \ \mbox{for any}\ f\in \fcd,
\end{equation*}
it is easy to see that $J$ is a non-trivial idempotent. Now we will show that $J$ lies in the bicommutant of $C_{\varphi}$.
Indeed, given any $A\in \{C_{\varphi}\}'$, $A^{\ast}\in \{C_{\varphi}^{\ast}\}'$. Then it is easy to see that $A\mathbb{1}\in \ker(C_{\varphi}-I)$ and $A^{\ast}k_w\in \ker(C_{\varphi}^{\ast}-I)$.
It follows that $A\mathbb{1}=\lambda\mathbb{1}$ and $A^{\ast}k_w=\mu k_w$ for some constant $\lambda,\mu\in \mathbb{C}$. Taking the scalar product,
$${\mu=\frac{\langle A^*k_w,\mathbb{1}\rangle}{\langle k_w, \mathbb{1}\rangle}=
\frac{\langle k_w,A\mathbb{1}\rangle}{\langle k_w,\mathbb{1}\rangle}=\bar \lambda.}$$
Finally, noting that
\begin{equation*}
  A J=(A\mathbb{1})\otimes k_w=\lambda J
\end{equation*}
and
\begin{equation*}
JA=\mathbb{1}\otimes(A^{\ast}k_w)=\lambda J,
\end{equation*}
then we have $J \in\{C_{\varphi}\}''$.
\end{proof}

Now we can illustrate that on $\fcd$, any bounded composition operator is strongly reducible, and thus any bounded composition operator fails to possess the functional commutant property.

\begin{theorem}\label{fc-ness}
Any bounded composition operator is strongly reducible on $\fcd$.
\end{theorem}

\begin{proof}
Write $\varphi(z)=Az+B\in \mathcal{S}_{c}(\mathbb{C}^{d})$ with $A\in \mathcal{M}_{d}(\mathbb{C})$ and $B\in \mathbb{C}^{d}$. We can assume $A\neq 0$ otherwise $C_\varphi f(z)=f(B)$ and the statement is trivial.
By \cite[Theorem 4.2]{JPZ}, the spectrum of $C_\varphi$ is equal to the closure of
$$\left\{\lambda_1^{\alpha_1}\cdots \lambda_d^{\alpha_d}:\ \alpha\in\ZZ_+^d\right\}$$
where $\lambda_1,\dots,\lambda_d$ are the eigenvalues of $A.$ If $\sigma(A)\cap\DD\neq\emptyset,$ it is easy to see that this last set is not connected, since for any $r>0,$
$$\left \{|\lambda_1^{\alpha_1}\cdots \lambda_d^{\alpha_d}|:\ \alpha\in\ZZ_+^d,\ |\lambda_1^{\alpha_1}\cdots \lambda_d^{\alpha_d}|\geq r\right\}$$
is finite. By  \cite[Proposition 2.34]{JW}, $C_\varphi$ is not strongly irreducible.

Finally, if $\sigma(A)\subset\TT,$ then $B=0$ and $C_\varphi$ is a unitary diagonal operator, which is far from being irreducible.
\end{proof}

Finally, following from \cite[Lemma 2.2]{Ja}, if $T\in\mathcal{L}(\mathcal{X})$ has the functional commutant property, then $\{T\}'$ is abelian and $\{T\}'=\{T\}''$. But its converse does not hold generally by the following example.

\begin{example}
Let $A$ be a normal matrix in $\mathcal{M}_{d}(\mathbb{C})$ with $\|A\|\leq 1$ and $\sigma(A)=\{\lambda_{1},\cdots,\lambda_{d}\}$ with independent numbers $\lambda_1,\dots,\lambda_d$. Then $\{C_{A}\}'$ is abelian and $\{C_{A}\}'=\{C_{A}\}''$, but $C_{A}$ fails to have the functional commutant property.
\end{example}

\begin{proof}
By Lemma \ref{FCP-spectra} and Theorem \ref{fc-ness}, it suffices to show that $\{C_{A}\}'$ is abelian and $\{C_{A}\}'=\{C_{A}\}''$. For this, note that $C_{A}$ is similar to the diagonal operator $\Lambda:$
$z^{\alpha}\mapsto\lambda^{\alpha}z^{\alpha}$ for all $\alpha\in \mathbb{Z}_+^{d}$.
Since $\lambda_1,\dots,\lambda_d$ are independent, any $T\in \{\Lambda\}'$ also is a diagonal operator, i.e., there exists a sequence of complex numbers $t_{\alpha}$ such that $|t_{\alpha}|\leq\|T\|$ and
$T z^{\alpha}=t_{\alpha}z^{\alpha}$ for each $\alpha\in \NN^{d}$.
Indeed, one can see that $\langle T z^{\alpha}, z^{\beta}\rangle=0$ for any $d$-tuple indexes $\alpha\neq\beta$.
So $\{\Lambda\}'$ is abelian and $\{\Lambda\}'=\{\Lambda\}''$,
which in turn implies that $\{C_{A}\}'$ is abelian and $\{C_{A}\}'=\{C_{A}\}''$.
\end{proof}

\section{Mean ergodic composition operators}

In this section, we apply the results of Section \ref{sec:iteration} to characterize mean and uniformly mean ergodic composition operators on $\mathcal F(\CC^d)$.
We recall that $T\in\mathcal{L}(\mathcal{X})$ is said to be \emph{mean ergodic} if there exists $P\in\mathcal{L}(\mathcal{X})$ on a Banach space $\mathcal{X}$ such that the $n$-th ergodic mean of $T$
\begin{equation*}
  M_{n}(T):=\frac{1}{n}\sum_{k=1}^{n}T^{k}
\end{equation*}
converges to $P$ in the strong operator topology,
and \emph{uniformly mean ergodic} if $M_{n}(T)\rightarrow P$ in the uniform operator topology.

The first step to characterize mean ergodic composition operators on $\mathcal F(\CC^d)$ is to prove that they are always power bounded. Recall that an operator $T$ on a Banach space is \emph{power-bounded} if $\sup_{n\geq1}\|T^{n}\|<\infty$.

\begin{theorem}\label{powerbounded}
Any bounded composition operator on $\fcd$ is power-bounded.
\end{theorem}

\begin{proof}
Let $\varphi(z)=Az+B\in \mathcal{S}_{c}(\mathbb{C}^{d})$. By Lemma \ref{simility}, we have $C_\varphi=C_{P^*}\circ C_{\phi}\circ C_{P}$ with $\phi(z)=Sz+v$ and use the same notation therein. So $\|C_{\varphi}^{n}\|=\|C_{\phi}^{n}\|$ for any non-negative integer $n$, and then it suffices to show that $C_{\phi}$ is power-bounded. For this, note that its $n$-th iteration is
\begin{equation*}
\phi^{[n]}(z)=S^{n}z+v^{(n)},\ \mbox{with}\ v^{(n)}:=\sum_{k=0}^{n-1}S^{k}v\in \mathbb{C}^{p}\times\{0\}^{d-p}.
\end{equation*}
Now it follows from (\ref{norm}) that
\begin{equation*}
  \|C_{\phi}^{n}\|=\|C_{\phi^{[n]}}\|=\exp\left(\frac{|z_0^{(n)}|^2 -|S^{n}z_0^{(n)}|^2+|v^{(n)}|^2}{4}\right),
\end{equation*}
where $z_{0}^{(n)}$ is any solution of $(I_{d}-(S^{n})^{\ast}S^{n})z=(S^{n})^{\ast}v^{(n)}$.

For simplicity, write $z=(z',z'')$ with $z'\in \mathbb{C}^{p}$ and
$z''\in\mathbb{C}^{d-p}$. Note that $|v^{(n)}|=|v^{(n)'}|$,
\begin{equation*}
  (S^{n})^{\ast}S^{n}=\begin{pmatrix} (T^{n})^{\ast}T^{n} & 0\\
  0 & (U^{n})^{\ast}U^{n} \end{pmatrix}=
  \begin{pmatrix} (T^{n})^{\ast}T^{n} & 0\\
  0 & I_{d-p} \end{pmatrix}.
\end{equation*}
Thus $(I_{d}-(S^{n})^{\ast}S^{n})z_{0}^{(n)}=(S^{n})^{\ast}v^{(n)}$ is equivalent to
$(I_{p}-(T^{n})^{\ast}T^{n})z_{0}^{(n)'}=(T^{n})^{\ast}v^{(n)'}$, and
\begin{equation*}
  \big|z_0^{(n)}\big|^2 -\big|S^{n}z_0^{(n)}\big|^2
  =  \big\langle (I_{d}-(S^{n})^{\ast}S^{n})z_{0}^{(n)}, z_{0}^{(n)}\big\rangle
  = \big\langle (I_{p}-(T^{n})^{\ast}T^{n})z_{0}^{(n)'}, z_{0}^{(n)'}\big\rangle.
\end{equation*}
Since the spectral radius of $T$ satisfies that $r(T)<1$ by Lemma \ref{simility}, then there is some positive integer $N$ such that $\|T^{n}\|<\frac{1}{2}$ for any $n>N$. In the sequel, let us restrict $n>N$, and get
\begin{equation*}
  \|(T^{n})^{\ast}T^{n}\|\leq\frac{1}{4},
\end{equation*}
which in turn implies that $I_{p}-(T^{n})^{\ast}T^{n}$ is invertible. Also we have
\begin{equation*}
\|(I_{p}-(T^{n})^{\ast}T^{n})^{-1}\|\leq\frac{1}{1-\|(T^{n})^{\ast}T^{n}\|}.
\end{equation*}
Therefore,
\begin{align*}
   \big|z_0^{(n)}\big|^2 -\big|S^{n}z_0^{(n)}\big|^2 = & \big\langle (T^{n})^{\ast}v^{(n)'},(I_{p}-(T^{n})^{\ast}T^{n})^{-1} (T^{n})^{\ast}v^{(n)'}\big\rangle  \\
   \leq & \big\|(I_{p}-(T^{n})^{\ast}T^{n})^{-1}\big\|\cdot \big|(T^{n})^{\ast}v^{(n)'}\big|^{2}\\
   \leq & \frac{\|T^{n}\|^{2}}{1-\|(T^{n})^{\ast}T^{n}\|} \big|v^{(n)'}\big|^{2},
\end{align*}
and then
\begin{equation*}
\|C_{\phi}^{n}\|\leq\exp\left(\frac{1}{4}\left(1+\frac{\|T^{n}\|^{2}}{1-\|(T^{n})^{\ast}T^{n}\|}\right)\big|v^{(n)'}\big|^{2}\right)
\leq \exp\left(\frac{\big|v^{(n)'}\big|^{2}}{3}\right).
\end{equation*}

Furthermore, $I_{p}-T$ is invertible and $\lim_{n\rightarrow\infty}T^{n}=0$, again by $r(T)<1$. Then
\begin{equation*}
\big|v^{(n)'}\big|\leq \left\|\sum_{k=0}^{n-1}T^{k}\right\||v'|\leq \|(I_{p}-T)^{-1}\|\|I_{p}-T^{n}\||v'|,
\end{equation*}
and thus
\begin{equation*}
\|C_{\phi}^{n}\|\leq\exp\left(\frac{\left(\|(I_{p}-T)^{-1}\|\|I_{p}-T^{n}\|\right)^{2}|v'|^{2}}{3}\right),
\end{equation*}
where the right hand side converges to $\exp\left(\frac{\|(I_{p}-T)^{-1}\|^{2}|v'|^{2}}{3}\right)$ as $n\rightarrow\infty$.
So $C_{\phi}$, and then $C_{\varphi}$, is power-bounded.
\end{proof}

It is well-known that every periodic operator is uniformly mean ergodic. But we need the following generalization, which is probably known, and we give a brief proof for completeness.

\begin{lemma}\label{lem:near-perodic}
Let $\mathcal{X}$ be a Banach space and $T\in\mathcal{L}(\mathcal{X})$. If there are some $m\in\NN$ and $P\in\mathcal{L}(\mathcal{X})$ such that $T^{mn}$ converges to $P$ in $\mathcal{X}$-norm as $n\to\infty$, then $T$ is uniformly mean ergodic. More precisely,
\begin{equation*}
\lim_{n\to\infty}\left\|M_{n}(T)-M_{m}(T)P\right\|=0.
\end{equation*}
\end{lemma}

\begin{proof}
Since $T^{mn}$ converges to $P$ in $\mathcal{X}$-norm as $n\to\infty$, then it follows immediately that the Ces\`{a}ro means of the
powers of $T^{m}$ also converge in operator norm to $P$; that is,
\begin{equation*}
\frac{1}{\ell}\sum_{q=0}^{\ell-1}T^{mq} \stackrel{\|\cdot\|}{\longrightarrow}P \text{ as } \ell\rightarrow\infty.
\end{equation*}

Now for any $n\in \NN$, write $n=m\ell+j$ with $j=0,1,\cdots,m-1.$
Then
\begin{align*}
M_{n}(T)&=
\frac{1}{m\ell+j}\sum_{k=1}^{m}\sum_{q=0}^{\ell-1}T^{mq+k}+
\frac{1}{m\ell+j}\sum_{k=1}^{j}T^{m\ell+k} \\
&=\frac{\ell}{m\ell+j}\sum_{k=1}^{m}T^k\left(\frac{1}{\ell}\sum_{q=0}^{\ell-1}T^{mq}\right)
+\frac{1}{m\ell+j}\sum_{k=1}^{j}T^{k}T^{m\ell}.
\end{align*}
Also noting that $\lim_{\ell\rightarrow\infty}\frac{\ell}{m\ell+j}=\frac{1}{m}$, then the first term satisfies that
\begin{equation*}
\frac{\ell}{m\ell+j}\sum_{k=1}^{m}T^k\left(\frac{1}{\ell}\sum_{q=0}^{\ell-1}T^{mq}\right)
\stackrel{\|\cdot\|}{\longrightarrow} \frac{1}{m}\sum_{k=1}^{m}T^kP,
\end{equation*}
as $\ell\rightarrow\infty$. On the other hand, since the sequence $(T^{m\ell})_{\ell\in\NN}$ is bounded, then the second term can be controlled by
\begin{equation*}
\left\|\frac{1}{m\ell+j}\sum_{k=1}^{j}T^{k}T^{m\ell}\right\|
\leq \frac{C}{m\ell+j} \sum_{k=1}^{m}\|T^k\| \rightarrow 0,\mbox{ as } \ell\rightarrow\infty,
\end{equation*}
where $C:=\sup_{\ell\geq 0}\|T^{m\ell}\|<\infty.$
Therefore, $M_{n}(T)\stackrel{\|\cdot\|}{\longrightarrow} M_{m}(T)P$,
as we wanted.
\end{proof}

Now the mean ergodicity of composition operators on $\fcd$ can be characterized.

\begin{theorem}\label{thm:mean-ergodic}
Any bounded composition operator is mean ergodic on $\fcd$. Furthermore, given any $\varphi(z)=Az+B\in\sccd$, $C_{\varphi}$ is uniformly mean ergodic if and only if either  $\sigma(A)\cap(\TT\backslash\{1\})=\varnothing,$ or all the elements in $\sigma(A)\cap(\TT\backslash\{1\})$ are roots of unity.
\end{theorem}

\begin{proof}
The first statement follows immediately from Theorem \ref{powerbounded} and the fact that every power bounded operator on a reflexive Banach space is mean ergodic (see for example \cite[\S 2.1, Theorem 1.2]{Kr}).

If $\sigma(A)\cap(\TT\backslash\{1\})=\varnothing,$ it is easy to see that $C_\varphi$ is uniformly mean ergodic with the limit $C_\psi$ by Theorem \ref{thm:iteration} (b).
If all the elements in $\sigma(A)\cap(\TT\backslash\{1\})$ are roots of unity, then there is a positive integer $m$ such that $\sigma(A^{m})\cap(\TT\backslash\{1\})=\varnothing$. So $C_{\varphi}^{m}=C_{\varphi^{[m]}}$ satisfies the condition (b) of Theorem \ref{thm:iteration}, and then $C_{\varphi}^{mn}$ converges to $C_{\phi}$ in norm for some $\phi\in \mathcal{S}_{c}(\CC^{d})$, as $n\rightarrow\infty$. Thus in this case $C_{\varphi}$ is uniformly mean ergodic by Lemma \ref{lem:near-perodic}.

Conversely, if there is some $\lambda\in\sigma(A)\cap(\TT\backslash\{1\})$ and $\lambda$ is not a root of the unit, then by \cite[Theorem 4.2]{JPZ}, $1\in\sigma(C_\varphi)$ and $1$ is an accumulation point of $\sigma(C_\varphi)$. Thus by \cite[\S 2.2, Theorem 2.7]{Kr}, $C_\varphi$ is not uniformly mean ergodic.
\end{proof}

Consequently, there exist some bounded composition operators being mean ergodic but not uniformly mean ergodic, similar to the phenomenon on the spaces of single-variable functions \cite{BGJJ, SMB}.

\medskip

\noindent \textbf{AI disclosure.}
We use AI tools at the very last step to help improve the language. All research ideas, strategies, and
mathematical arguments were developed independently by the authors.

\smallskip

\noindent \textbf{Acknowledgment.}
M. Wang and X. Yao were supported by National Science Foundation of China
(No. 12571145). The authors would like to thank Professor Chunlan Jiang for helpful comments about the reducibility.

\end{document}